\documentclass[11pt]{article}
\usepackage[utf8]{inputenc}

\usepackage{amsthm}
\usepackage{amssymb}
\usepackage{dsfont}
\usepackage{stmaryrd}
\usepackage{amsmath}
\usepackage{wrapfig}
\usepackage{overpic}

\newtheorem{theorem}{Theorem}
\newtheorem{conj}{Conjecture}
\newtheorem{lemma}[theorem]{Lemma}
\newtheorem{proposition}{Proposition}

\newtheorem{ques}{Question}
\newcommand{ \field }{ \mathbb{F}_p[t, t^{-1}] }
\newcommand{\F}{\mathbb{F}}

\newcommand{\N}{\mathbb{N}}
\newcommand{\Z}{\mathbb{Z}}

\DeclareMathOperator{\RF}{RF}
\DeclareMathOperator{\GL}{GL}
\DeclareMathOperator{\LCM}{LCM}

\title{Residual finiteness growths \\of Lamplighter groups}
\author{Khalid Bou-Rabee \and Junjie Chen \and Anastasiia Timashova}

\date{June 2019}

\begin{document}

\maketitle

\begin{abstract}
    Residual finiteness growth gives an invariant that indicates how well-approximated a finitely generated group is by its finite quotients.
    We briefly survey the state of the subject. We then improve on the best known upper and lower bounds for lamplighter groups. Notably, any lamplighter group has super-linear residual finiteness growth. In our proof, we quantify a congruence subgroup property for lamplighter groups.
\end{abstract}

\section{Introduction}

Recall that a group is \emph{residually finite} if the intersection of all its normal finite-index subgroups is trivial. This property, which is closed under direct products, passing to subgroups, and passing between commensurable groups, has a number of equivalent notions:
\begin{eqnarray*}
&\hskip.1in& \bigcap_{[G:\Delta] < \infty, \Delta \lhd G } \Delta = \{ 1 \}
\iff \bigcap_{[G:\Delta] < \infty, \Delta \leq G } \Delta = \{ 1 \} \\
&\iff& \forall g \in G \setminus \{1 \}, \exists \text{ a homomorphism } \phi: G \to Q,\; Q \text{ finite}, \phi(g) \neq 1 \\
&\iff&  G \text{ acts by automorphisms on a rooted locally finite tree.}
\end{eqnarray*}
For an element $g \in G$ where there is a homomorphism $\phi: G \to Q$ with $\phi(g) \neq 1$, we say $\phi$ or $\ker\phi$ \emph{detects} $g$, or that $g$ \emph{survives} through $\phi$.

Residual finiteness is a fundamental property that all presently known word-hyperbolic groups and all finitely generated linear groups possess.\footnote{Groups that are not residually finite include: the rational numbers, the non-solvable Baumslag-Solitar groups ($\left< a, b : a b^p a^{-1} = b^q \right>$), abstract commensurators of surface groups, and all infinite simple groups.} Let $B_{\Gamma, X} (n)$ denote the metric $n$-ball in $\Gamma$ with respect to the word metric $\| \cdot \|_{\Gamma, X}$.
To quantify residual finiteness, \cite{B09} posed the following question:

\begin{ques} \label{quantifyresidualfinitenessquestion}
How large of a group do we need to detect elements in $B_{\Gamma, X}(n)$? 
For instance, what is the smallest integer $\RF_{\Gamma,X}(n)$ such that each nontrivial element in $B_{\Gamma,X}(n)$ is detected by a finite group of cardinality no greater than $\RF_{\Gamma, X}(n)$?
\end{ques}

There are a number of natural ways to view the function $\RF_{\Gamma,X}$. In algebra, it gives a quantification of some congruence subgroup property for the group.
In geometry, it measures the interaction of the word metric and the profinite metric. In decision theory, it quantifies how the word problem may be solved through finite quotients.
In topology, it determines how large an index of a regular cover we need to lift a given closed loop to a non-closed loop.


For a group $\Gamma$ and any subset $S \subset \Gamma$, let $S^\bullet$ denote the set of all non-identity elements in $S$.
Question \ref{quantifyresidualfinitenessquestion} may be approached through the study of the asymptotic behavior of the \emph{divisibility function} (aka the \emph{valuation function}) $D_\Gamma : \Gamma^\bullet \to \N$ defined by
$$
D_\Gamma(\gamma) = \min \{ [ \Gamma : \Delta] : \gamma \notin \Delta, \Delta \lhd \Gamma \}.
$$
The integer $\RF_{\Gamma, X}(n)$ is the maximum value of $D_\Gamma$ on $B^\bullet_{\Gamma, X}(n)$.
Groups for which $D_\Gamma |_{\Gamma^\bullet}$ take on finite values are precisely residually finite groups.
The statistics of divisibility functions is the subject of quantifying residual finiteness, a notion introduced in \cite{B09}. Related functions have been studied in \cite{MR3430831, MR3570303, MR3530864, BD14}.
Rivin also did some analysis on this question by giving a (conjecturally) sharp result graded by the depth of an element in the lower central series, as well as ``ungraded'' results \cite{R12}.

These functions have been computed for a large selection of examples by numerous people, please see Table \ref{QRFTable}. 
Keep in mind that large values for $\RF_{\Gamma,X}$ indicate that the group is not well approximated by finite quotients.

\begin{table}[h]
\setlength\tabcolsep{2pt}
\label{QRFTable}
\centering
\begin{tabular}{|c|c|c|c|}  
\hline
group $G$ & dominated by & dominates & citation \\ [0.5ex] 
\hline\hline 
fg infinite & nothing & $ \log\log(n) $ & \cite{BS13b}, \cite{BM13} \\
[0.5ex]
fp solvable & no definable & 1 & \cite{MR3671739} \\ 
 &  function & & \\ [0.5ex]
fg infinite linear &$n^{k_G}$ & $\log(n)$ & \cite{BM13} \\ [0.5ex] 
higher rank arithmetic & $n^{\dim G}$ & $n^{\dim G}$ & \cite{BK12} \\ [0.5ex] 
higher rank $G(\F_p[t])$ & $n^{\dim G}$ & $n^{\dim G}$ & \cite{MR3633299} \\ [0.5ex]
fg nonabelian free & $n^{3}$ & $\frac{n^{3/2}}{\log(n)^{9/2+\epsilon}}$ & \cite{B09}, \cite{MR3937332} \\ [1.2ex] 
{\bf lamplighter} & $n^2$ &  $n^{3/2}$ &{\bf this paper} \\ [0.5ex]
fg infinite abelian group & $\log(n)$ & $\log(n)$ & \cite{B09} \\ [0.5ex] 
fg infinite virtually nilpotent & $[\log(n)]^{k_G}$ & $\log(n)$ & \cite{B09} \\ [0.5ex] 
the first Grigorchuk group & $2^n$ & $2^n$ &  \cite{B09} \\ [0.5ex]
the Gupta-Sidki $p$-group &  & $2^{n^{\frac{\log(p)}{\log(3)}}}$ &  \cite{MR3704931} \\ [0.5ex]
the Pervova group &$2^{n^{c}} $ & $2^{n^{d}} $ &  \cite{MR3704931} \\ [0.5ex]
\hline
\end{tabular}
\caption{Bounds for $\RF_{\Gamma, X}(n)$. We use \emph{fg} to abbreviate ``finitely generated'' and \emph{fp} to abbreviate ``finitely presented''.}
\label{table:nonlin} 
\end{table}
\noindent

\subsection{Our main result}
In this paper, we improve the best known upper and lower bounds for lamplighter groups. These groups are defined in terms of wreath products, see \S \ref{subsec:reps} for their definitions.
Before stating our contribution, we fix some notation: $f \preceq g$ means that there exists $C > 0$ such that $f(n) \leq C g(C n)$, and $f \asymp g$ means that $f \preceq g$ and $g \preceq f$. When $f \preceq g$ we say that $g$ \emph{dominates} $f$, or that $f$ is \emph{dominated by} $g$. The growth of $\RF_{\Gamma, X}(n)$ is, up to $\asymp$ equivalence, unchanged by the choice of $X$ \cite{B09}. Thus, when asymptotic growth is computed, the subscript $X$ is often dropped from $\RF_{\Gamma, X}(n)$.

\begin{theorem} \label{thm:one}
Let $p$ be a prime and let $L := \Z/p \wr \Z$. Then 
$$n^{3/2}  \preceq \RF_L(n) \preceq n^2.$$
\end{theorem}
\noindent

We conclude the remarkable fact that the best lower bounds and upper bounds closely match those for nonabelian free groups (see Table \ref{QRFTable}). That is, loosely speaking, with present tools, it seems just as difficult to use finite quotients to approximate a lamplighter group as it is to approximate a nonabelian free group. Our proof methods are elementary, showing that the classical representations of lamplighter groups have a congruence subgroup property that works well with the word-length geometry of the group. See \S \ref{subsec:reps} for the representations and \S \ref{sec:theproofs} for the proofs. In particular, the congruence subgroup property for lamplighter groups is stated in Proposition \ref{prop:csp} in \S \ref{sec:theproofs}.

We view Theorem \ref{thm:one} as a quantification of the congruence subgroup property for $L$ (see Proposition \ref{prop:csp} in \S \ref{sec:theproofs}) because if one were to redefine $\RF_L$ by restricting to finite quotients that come from principal congruence quotients (or even ``ensuing congruence quotients'', see the definition and discussion after Proposition \ref{prop:csp})  one would arrive at the same upper and lower bounds. That is, it seems just as efficient to view $L$ through all its finite quotients as it is to view $L$ through principal congruence quotients.

Prior to this paper, the best known bounds for lamplighter groups were found in \cite{B11}. There it is shown that $\RF_L(n) \succeq n^{1/2}$. While lamplighter groups are linear in positive characteristic, it is quite difficult to find optimal bounds for them. Lamplighter groups, as opposed to the other linear groups in Table \ref{QRFTable}, do not have cyclic subgroups that are exponentially distorted and, moreover, are far from hyperbolic. Exponential distortion is a robust tool for finding bounds in linear groups. Recently, such methods were used to find lower bounds for residual finiteness growths of finitely generated solvable groups that admit infinite order elements in the Fitting subgroup of strict distortion at least exponential \cite{MP19}. Moreover, all non-elementary hyperbolic groups are malabelian (of which the lamplighter group, and any solvable group, is not), which is a significant tool for finding lower bounds for hyperbolic groups \cite{BM11}.

In \S \ref{sec:furtherdirections}, we formulate a number theoretic conjecture that we hope will eventually lead to determining the residual finiteness growth of lamplighter groups.  



\paragraph{Acknowledgements}
We are grateful to Ahmed Bou-Rabee, Rachel Skipper, and Daniel Studenmund for giving us comments and corrections on an earlier draft.
This work is partially supported by NSF grant DMS-1820731 and PSC-CUNY Award \# 62134-00 50.

\section{Preliminaries}
\label{sec:prelims}

\subsection{Solvable groups}

Lamplighter groups are contained in the broad class of ``solvable groups''.
Let $\Gamma$ be a group. Set $\Gamma^{(k)}$ to be the  \emph{derived series} of $\Gamma$, defined recursively by
$$\Gamma^{(0)} = \Gamma \text{ and } \Gamma^{(k)} = [\Gamma^{(k-1)}, \Gamma^{(k-1)}].$$
A group $\Gamma$ is said to be \emph{solvable} if $G^{(k)} = 1$ for some natural number $k$.
The minimal such $k$ is called the \emph{solvable class} of $\Gamma$.
If, in addition to $\Gamma$ being solvable, each quotient $\Gamma^{(k)}/\Gamma^{(k+1)}$ is finitely generated, then $\Gamma$ is said to be \emph{polycyclic}.
Equivalently, a group $\Gamma$ is polycyclic if and only if $\Gamma$ is solvable and every subgroup of $\Gamma$ is finitely generated.
Set $\Gamma_k$ to be the lower central series for $\Gamma$, defined recursively by
$$\Gamma_0 = \Gamma \text{ and } \Gamma_k = [\Gamma_{k-1}, \Gamma].$$
A group $\Gamma$ is said to be \emph{nilpotent} if $\Gamma_k = 1$ for some natural number $k$. The minimal such $k$ is the \emph{nilpotent class} of $\Gamma$.

\subsection{Lamplighter groups}

\label{subsec:reps}
Let $A$ and $B$ be groups and $\Omega$ a set that $B$ acts on.
The \emph{wreath product}, denoted $A \wr B$, is defined to be the semidirect product 
$$
\left( \prod_{i \in \Omega} A_i \right) \rtimes B,
$$
\noindent
where each $A_i$ is a copy of $A$ and the action of $B$ on $\left( \prod_{i \in \Omega} A_i \right)$ is given by
$$
b ((a_i)_{i \in \Omega}) := (a_{b(i)})_{i \in \Omega}.
$$
Using this notation, \emph{$p$-Lamplighter group} is the wreath product $L := \Z/p \wr \Z$, where $\Omega = \Z$ and $\Z$ acts on $\Omega$ by addition. That is, for $b \in \Z$,
$$
b ((a_i)_{i \in \Omega}) := (a_{b+i})_{i \in \Omega}.
$$
When $p = 2$, each element of $\prod_{i \in \Omega} A_i$ is thought of as a state of an infinite line of lamps. Each coordinate signals whether a corresponding lamp is ``on'' or ``off'' with a 1 or 0. When $p > 2$, we think of the lamps as having different shades or colors. When it is clear what $p$ is from context we will simply call this semidirect product the \emph{lamplighter group}.

Each lamplighter group is a solvable group of class 2. Moreover, each contains an infinite direct product of copies of $\Z/p$, and hence is neither polycyclic nor nilpotent.
While the following two lemmas are likely well-known, we include them for completeness.

\begin{lemma} \label{lem:lamplighterlinear}
The group L is isomorphic to the group of matrices
$$
 L'= \left\{ \begin{pmatrix} t^k & p(t) \\ 0 & 1 \end{pmatrix} : k \in \Z, p(t) \in \F_p[t, t^{-1}] \right\}.
$$
\end{lemma}

\begin{proof}
Let $\phi: L \rightarrow L'$ be the map given by 
$$
\phi((a_i)_{-\infty}^{\infty},n):= \left(\sum_{-\infty}^{\infty}{a_{i}t^i},t^n\right).
$$
This map is visibly bijective.
To see that it is a homomorphism, we compute:
\begin{eqnarray*}
\phi((a_i)_{-\infty}^{\infty},n_1) \phi((b_i)_{-\infty}^{\infty},n_2)
&=&(\sum_{-\infty}^{\infty}{a_{i}t^i},n_1)(\sum_{-\infty}^{\infty}{b_{i}t^i},n_2)\\
&=&(\sum_{-\infty}^{\infty}{a_{i}t^i}+t^{n_1}\sum_{-\infty}^{\infty}{b_{i}t^i},n_1+n_2)\\
&=&\phi((a_i+b_{i-n_1})_{-\infty}^{\infty},n_1+n_2)\\
&=&\phi(((a_i)_{-\infty}^{\infty},n_1)((b_i)_{-\infty}^{\infty},n_2)).
\end{eqnarray*}
\end{proof}

We identify $L$ with its image in the representation above and denote any element of the form 
$$\begin{pmatrix} t^k & p(t) \\ 0 & 1 \end{pmatrix}$$ 
by $(p(t), k)$. 
Define $s_0 :=(1,0)$ and $t :=(0,1)$. We have $(s_0)^{p}=0$, and $L = \left< s_0, t \right>$.
The representation in Lemma \ref{lem:lamplighterlinear} gives us a lens to see the geometry of the Cayley graph of $L$ with respect to the generators $X$.

 \begin{lemma} \label{lem:wordlength}
Set $X = \{s_0, t \}$.
Let $p(t) = \frac{f(t)}{t^q}$, where $f(t)$ is a polynomial of degree $a$  with  nonzero constant term. Then the element $(p(t),k)$ has word length at most 
\[
\max\left\{2|q|+|k|+(p+1)a,(p+3)a+|k|\right\}
\]
with respect to $X$.
 \end{lemma}
 
 \begin{proof} Let $p(t) = \sum_{i=-q}^{a-q} {b_i x^i} $. We have that $t^{-q}s_0^{b_{-q}}=(b_{-q} t^{-q},q)$, and similarly 

\[
(b_{-q} t^{-q},-q) t s_0^{b_{-q+1}}=(b_{-q} t^{-q}+b_{-q+1} t^{-q+1},-q+1)
\]

Repeating this process for each term of $p(t)$, we obtain
$$
(p(t),k) = t^{-q-1}\left(\prod_{i=-q}^{a-q} ts_0^{b_i}\right) t^{k+q-a}
$$
with $b_i \leq p$ .

So, we have the word length is at most 
$$
 (q-1)+(p+1)(a+1)+|k+q-a|
$$
which is at most
$\max\{2|q|+|k|+(p+1)a,(p+3)a+|k|\}.$
\end{proof}


\section{The proofs}
\label{sec:theproofs}

\subsection{Upper bounds for lamplighter groups}
\label{subsec:upperboundL}
\begin{proposition}
$\RF_L(n) \preceq n^{2}$. 
\end{proposition}

\noindent
We use the characteristic $p$ representation of the lamplighter group to show that the upper bound of $RF_n(L)$ is $n^{2}$.

\begin{proof}
Let $L$ be the $p$-lamplighter group realized as a subgroup of  $\GL_2(\F_p{(t)})$ as in Lemma \ref{lem:lamplighterlinear}.
Then $L$ is generated by two elements $s_0$ and $t$. By adding inverses we expand this to a generating set $X$ of size 4. 
The set of all entries of the matrices in $X$ consists of
$\{ 1,-1, t, t^{-1}, 0\}$.
It follows that $L$ is a subgroup of 
$\GL_2(\F_p[t, 1/t])$.

Let $g \in L \setminus \{ 1 \}$ be an element in $L$ of word length $n$ with respect to $X$. 
Then there exists a non-zero entry $\alpha$ of $g-1$, where $\alpha \in \F_p{(t)}$. 
By our selection of $g$, the element $\alpha$ is of the form 
$\frac{f(t)}{g(t)}$, where $f(t)$ and $g(t)$ are in $\F_p{[t]}$, and their degrees are bounded by $n$.

By \cite[Lemma 2.2]{MR3430831}, the polynomial $t f(t)$ survives in a quotient field of $\F_p[t]$ of cardinal $\leq 2(n+1)p$.
Hence, the element $g$ is detectable by the induced homomorphism $\phi: \GL_2(\F_p[t, t^{-1}]) \to \GL_2(\F_{q})$ where $q \leq 2(n+1)p$.

To finish, we need only compute an upper bound for the cardinal of $\phi(L)$. Elements in $\phi(L)$ are $2\times 2$ matrices with the lower two coordinates being $1$ and $0$. Since $q \leq 2(n+1)p$, this leaves at most $q^2$ options for elements in $\phi(L)$.
This gives the desired upper bound, as $q \leq 2(n+1)p$.
\end{proof}

\subsection{Lower bounds for lamplighter groups}

We need some technical results before writing our lower bounds.
For the following lemmas, we identify $\field$ with the normal subgroup consisting of elements of the form $(p(t), 1)$ in $L$.

\begin{lemma} \label{lem:ideal}
If $N \trianglelefteq L$, then $N \cap \field$ is an ideal in $\field$.
\end{lemma}

\begin{proof}
Notice that the lamplighter group is given by $L = \field \rtimes \Z$, which is the semi-product defined by the group homomorphism $\varphi : \Z \shortrightarrow Aut(\field)$ by  $\varphi(z)(p(t))=t^zp(t)$. 
By the definition of $L$, it follows that $\field \trianglelefteq L$. Since $N \trianglelefteq L$, it follows that $ N \cap \field \trianglelefteq L $. 
Let $N \cap \field = M$. Then for all $(p(t), 0) \in M$, we have $(0, k)(p(t), 0)(0, k)^{-1} = (0, k)(p(t), 0)(0, -k) = (t^{k}p(t), 0) \in M$. 
We also have $(p_{1}(t), 0)(p_{2}(t), 0) = (p_{1}(t)+p_{2}(t), 0)$, which means $M$ is closed under the addition defined normally. 
Hence, for every $(p(t), 0) \in M$ and $\sum_{-\infty}^{\infty} a{_k}\cdot t^{k} \in \field$, we have $p(t)\cdot\sum_{-\infty}^ {\infty} a{_k}\cdot t^{k} \in M$. Therefore, $M$ is ideal in $\field$.\\
\end{proof}

The following result from \cite{MR3278388} will be useful for a partial characterization of finite quotients of lamplighter groups. We include its proof for completeness. Notice that $\Z[t,t^{-1}]$ is \emph{not} a principal ideal domain.

\begin{lemma} \label{lem:pid}
$\field$ is a principal ideal domain.
\end{lemma}

\begin{proof}
Let M be an ideal in $\field$. Let $P = M  \cap  \F_p[t]$. 
Being the intersection of an ideal in $\field$ and an ideal in $\F_p[t]$, the set $P$ is itself an ideal in $\F_p[t]$. 
Since $\F_p[t]$ is a principal ideal domain, it follows that $P$ is generated by a single element, $\alpha(t) \in \F_p[t]$.
We show that $M =(\alpha(t))$. Every element in $M$ is of the form $\frac{f(t)}{(t^k)}$. Let $\frac{f(t)}{(t^k)}$ be an arbitrary element in $M$. By multiplying by $t^k$ we see that $f(t) \in P$.
Since $P$ is generated by $\alpha(t)$ as an ideal in $\F_p[t]$, we have $\frac{f(t)}{(t^k)} = \alpha(t) \beta(t) t^{-k}$ for some $\beta(t) \in \F_p[t]$.
Therefore, $M \subset (\alpha(t))$. Moreover, since $\alpha(t) \in M$ and $M$ is ideal, we have $(\alpha(t)) \subset M$.
Thus, $M=(\alpha(t))$. This implies $\field$ is a principal ideal domain.
\end{proof}





A finite-index normal subgroup $\Delta$ of $L$ is a \emph{principal congruence subgroup} if it is the kernel of a surjective homomorphism of the form
$$
r : L \to (\field/(g(t)))\rtimes \Z/k.
$$
A finite-index subgroup is a \emph{congruence subgroup} of $L$ if it contains a principal congruence subgroup. As a consequence of Lemma \ref{lem:ideal} we have what is commonly referred to as the \emph{congruence subgroup property} (c.f. \cite{MR3278388}, which contains much stronger results):

\begin{proposition} \label{prop:csp}
Any finite-index subgroup of $L$ is a congruence subgroup.
\end{proposition}

Before proving this proposition we remark that if principal congruence subgroups were instead defined to be kernels of homomorphisms
$$
r : L \to \GL_2( \field / (g(t)),
$$
that factor through the fixed representation $L \to \GL_2(\field / (g(t)))$ from \S \ref{sec:prelims},
then $L$ would {\bf not} have the congruence subgroup property. 
The reason is that there are infinitely many images of $(\field/g(t)) \rtimes \Z$ that are finite whereas there is only one image of $L$ in the fixed representation to $\GL_2(\field /g(t))$.
We call such kernels \emph{ensuing congruence subgroups}.

\begin{proof}[Proof of Proposition \ref{prop:csp}]
Let $N$ be a normal subgroup of finite index in $L$.
Then by Lemma \ref{lem:ideal} and Lemma \ref{lem:pid}, there exists $g \in \F_p[t]$ such that $N \cap \field = (g(t))$.
Letting $k$ be the order of $\left< t \right>/N \cap \left< t \right>$, we see that $N$ must contain the kernel of the homomorphism 
$$
r : L \to (\field/(g(t)))\rtimes \Z/k.
$$
We conclude that $N$ is a congruence subgroup, as desired.
\end{proof}

To help quantify Proposition \ref{prop:csp} and arrive at Theorem \ref{thm:one}, we need a combinatorial result concerning least common multiples of polynomials.
Let $P_d$ be the collection of all polynomials in $\F_p[t]$ of degree less than or equal to $d$.

\begin{lemma} \label{lem:degofw}
We have 
$$
p^{d} \leq \deg(\LCM(P_d)) \leq  2 p^{d+1}.
$$
\end{lemma}

\begin{proof}
Set $w := \deg(\LCM(P_d))$.
We first compute an upper bound for $w$. Let $f(t) := \LCM(P_d)$.
By \cite[Lemma 2.1]{MR3430831}, if $M(i)$ be the number of irreducible  polynomials of degree $i$, then $M(i) \leq \frac{p^{i+1}}{i}$ (the additional $p$ factor comes in because the cited lemma counts the number of \emph{monic} irreducible polynomials).
For each irreducible polynomial $g(x)$ of degree $i$, we have that [$g(x)^{k_i}$ divides $f(x)$] and [$g(x)^{k_i+1}$ does not divide $f(x)$] for $k_i$ satisfying $k_i i \leq d < (k_i+1)i$.
Thus, each irreducible polynomial of degree $i$ contributes at most $k_i i$ to the overall degree of $f(t)$.
Hence, the degree of $f(t)$ is bounded above by $\sum_{i=1}^d i k_i \frac{p^{i+1}}{i} \leq \sum_{i=1}^d \frac{d}{i} p^{i+1}.$
Hence,
$$
\deg f(t) \leq  d \sum_{i=1}^d  \frac{p^{i+1}}{i}.
$$
Let $S_r := r \sum_{i=1}^r  \frac{p^{i}}{i}$.
First, we will show that for $p > 2$, $S_r \leq 2 p^r$ by induction.
The base case is clear.
For the inductive step, we compute

\begin{eqnarray*}
S_r &=& r [ p + p^2/2 + \cdots + p^r/r] \\
&=& r \left[ \frac{S_{r-1}}{r-1} \right] + p^r \\
&\leq& \frac{2r}{(r-1)p} p^r + p^r.
\end{eqnarray*}
We have $\frac{2r}{(r-1)p} \leq 1$ for $ p > 2$, and so the desired inequality $S_r \leq 2 p^r$, and hence $\deg f(t) \leq p S_d \leq 2  p^{d+1}$, follows.



For $p = 2$, one can use the following induction on $r$ for $S_r \leq 3 \times 2^r \iff \sum_{k=1}^r 2^k/k \leq 3 \times 2^r/r$:
The base case $r=1$ follows from a quick computation. For the inductive step where $r > 1$,
$$
\sum_{k=1}^{r-1} \frac{2^k}k \leq 3\times  \frac{2^{r-1}}r + \frac{2^r}r
= \frac{5r-2}{2r-2} \frac{2^r}r < 3 \times \frac{2^r}r.
$$
This gives the desired upper bound, since $M(i) \leq p^{i}/i$ and so $\deg f(t) \leq S_d$ in the case $p = 2$. Indeed, $\deg f(t) \leq S_d \implies \deg f(t) \leq 3 \times 2^d \leq 4 \times 2^d = 2 \times 2^{d+1}.$


We next prove a lower bound for $w$.
Notice that the product of all irreducible polynomials of degree dividing $d$ is precisely
$$
x^{p^d} - x,
$$
hence $w$ is bounded below by at least $p^d$, as desired.
\end{proof}

\begin{lemma} \label{lem:polydegree}
Let $p(t) = \LCM( P_{d})\prod_{i=1}^{\sqrt{w}}(1-t^i)$.
Then, the degree of $p(t)$ is bounded above by $C p^d$ for some constant $C > 0$ that is independent of $d$.
\end{lemma}

\begin{proof}
Lemma \ref{lem:degofw} gives that $\deg(\LCM(P_d))$ is bounded above by $N p^d$, where $N > 0$ is independent of $d$.
The polynomial $\prod_{i=1}^{\sqrt{w}}(1-t^i)$ has degree $M w$, where $M$ does not depend on $p$.
Setting $C = N + M$, we get that the degree of $p(t)$ is bounded above by $C p^d$, as desired.
\end{proof}

\begin{theorem} \label{thm:upperbound}
$\RF_L(n) \succeq n^{3/2}$.
\end{theorem}

\begin{proof}
Recall elements in $L$ are of the form $(p(t),k)$, and $L$ is generated by the elements $s_0:=(1,0)$ and $t:=(0,1)$.
Further recall that $P_d$ is the set of all polynomials of degree less than $d$.
Consider the element $(p(t),0)$ such that 
$$p(t) = \LCM( P_{d})\prod_{i=1}^{\sqrt{w}}(1-t^i)$$ 
where $w$ is $p^d$. 
By Lemmas \ref{lem:wordlength} and \ref{lem:polydegree} the word length of $(p(t), 0)$ is at most $C p^{d}$, where $C > 0$ is a constant independent of $d$.

Let $N$ be a normal subgroup of $L$ that detects the element $(p(t), 0)$.
By Lemma \ref{lem:pid}, we have $I := N \cap \field = (\alpha(t))$ for some $\alpha(t) \in \F_p[t]$.
Since $N$ detects $(p(t), 0)$ it follows that $p(t) \notin N \cap \field$.
Hence, $\alpha(t)$ must be a polynomial of degree greater than $d$.
It follows that the cardinality of $I$ under the quotient map $\pi: L \to L/N$ is greater than $p^d$. Call the image of $I$ under this quotient map $Q_I$.

Consider the image of $t = (0,1)$ under the quotient map $\pi$.
Let $\Z/k\Z$ be the cyclic group isomorphic to $\left< \pi((0,1)) \right>/Q_I$.
Then, by construction, $\alpha(t)$ must divide $1-t^k$. But $\alpha(t)$ cannot divide $\prod_{i=1}^{\sqrt{w}}(1-t^i)$, so it follows that $k \geq \sqrt{w}$, and we have by Lemma \ref{lem:degofw}, that $\sqrt{w} \geq m \sqrt{d p^d}$ for some $m > 0$ independent of $d$.
Therefore, 
$$
k \geq m \sqrt{p^d}.
$$
Thus, we obtain the total lower bound
$$
[L:N] \geq m \sqrt{p^d} p^d = m p^{3d/2}.
$$



Hence, we have showed that $\RF_L( C p^{d}) \geq  m p^{\frac{3d}{2}-1}$ where both $C$ and $m$ do not depend on $d$. Now we will infer $\RF_L(n)\geq n^{3/2}$ from it.
For every $n$, we have such $d$, such that $p^{d+1}\geq n \geq p^{d}$. Since $\RF_L(n)$ is increasing function, we have 
$$\RF_L(C n) \geq C'  p^{\frac{3d}{2}}= C' n^{3/2},$$
as desired.
\end{proof}

\section{Further directions}

\label{sec:furtherdirections}
\begin{ques}
Does any lamplighter group have the same residual finiteness growth of any nonabelian free group?
\end{ques}

In order to reasonably tackle this question, one would have to narrow down the residual finiteness growths of both classes of groups. For the lamplighter group, given the nice nature of the lower bound candidates, we guess that the residual finiteness growth is $n^{3/2}$ (which matches the conjecture for the residual finiteness growth of non-abelian free groups). Assuming the growths for lamplighter groups are $n^{3/2}$ we can formulate the following number theoretic conjecture.

Before stating the conjecture, we fix some notation.
We call any polynomial in $\F_p[x]$ of the form $x^d-1$ a \emph{periodic polynomial}. This name is motivated by the following: if the image of $t^d - 1$ in $\field \leq L$ is zero in a finite quotient, then that means that any configuration of lamps in the kernel is periodic with period-length dividing $d$.
A polynomial $g(x) \in \F_p[x]$ is \emph{almost-periodic} if $g(x)$ divides $x^d - 1$ where $d^2 \leq p^{\deg(g)}$.
Let $A_k$ be the collection of almost-periodic polynomials in $\F_p[x]$ of degree less than or equal to $k$.
Any periodic polynomial is almost-periodic, so  $\deg(\LCM(A_k))$ has a linear lower bound in terms of $k$.

\begin{conj} \label{conj:1}
There exists $C > 0$, that does not depend on $k$ such that
$$\deg(\LCM(A_k)) \geq C p^k.$$
\end{conj}

\noindent

If Conjecture \ref{conj:1} is true, then one can use principal congruence quotients to produce the upper bound $n^{3/2}$ for the residual finiteness growth of any lamplighter group.

\bibliography{refs}
\bibliographystyle{amsalpha}












\end{document}